\documentclass[11pt]{amsart}
\usepackage{amsmath,amsfonts,amsthm,amscd,amssymb,graphicx}
\numberwithin{equation}{section}

\newtheorem{theorem}{Theorem}[section]

\newtheorem{lemma}[theorem]{Lemma}

\newtheorem{proposition}[theorem]{Proposition}

\newtheorem{definition}[theorem]{Definition}

\def\eps{\varepsilon }
\def\D{\partial }

\newcommand{\RR}{\mathbb{R}}
\newcommand{\cO}{\mathcal{O}}

\newcommand{\CC}{\mathbb{C}}

\newcommand{\C}{\mathbb{C}}

\newcommand{\dt}{\frac{d}{dt}}

\newcommand{\NN}{{\mathbb N}}
\newcommand{\ZZ}{{\mathbb Z}}
\newcommand{\TT}{{\mathbb T}}

\newcommand{\e}{{\varepsilon}}

\def\bb1{{1\!\!1}}

\def\cL{\mathcal{L}}

\def\R{\Re e}
\def\I{\Im m}

\def\w{{\omega}}

\newcommand{\T}{{\mathbb T}}
\newcommand{\pa}{{\partial}}

\begin{document}


\title[Remarks on the ill-posedness of the Prandtl equation]{Remarks on the ill-posedness \\ of the Prandtl equation}

\author[D. G\'erard--Varet, T. Nguyen]{David G\'erard--Varet, Toan Nguyen}

\thanks{The first author acknowledges the support of  ANR project ANR-08-JCJC-0104 - CSD 5.
 The second author was supported by  the Foundation Sciences Math\'ematiques de Paris under a postdoctoral fellowship.}

\address{Equipe d'analyse fonctionnelle,
Universit\'e Denis Diderot (Paris 7),
Institut de Math\'ematiques de Jussieu, UMR CNRS 7586.
 }

\email{gerard-varet@math.jussieu.fr}

\address{Equipe d'analyse fonctionnelle,
Universit\'e Pierre et Marie Curie (Paris 6),
Institut de Math\'ematiques de Jussieu, UMR CNRS 7586.
 (Current address: Division of Applied Mathematics, Brown University, Providence, RI, USA)}

\email{nguyent@math.jussieu.fr}

\begin{abstract} In the lines of the recent paper \cite{GD},  we establish various ill-posedness results for the Prandtl equation.  By considering perturbations of stationary shear flows, we show that for some linearizations of the Prandtl equation and some  $C^\infty$ initial data,  local in time  $C^\infty$ solutions do not exist.  At the nonlinear level, we prove that if a flow exists in the Sobolev setting,  
it cannot be
Lipschitz continuous. Besides  ill-posedness in time, we also establish some ill-posedness in space, that casts some light on the results obtained by Oleinik for monotonic data. 
\end{abstract}

\date{\today}

\maketitle


\section{Introduction}
Our concern in this paper is the famous Prandtl equation: 
\begin{equation}\label{prandtl}
\left\{
\begin{array}{rlll} 
\D_t u  +  u \D_x u +  v \D_y  u - \D^2_y  u  + \D_x P &=& f,  \quad y > 0,\\
\D_x  u + \D_y v & = & 0,  \quad  y > 0, \\
u = v  &=& 0, \quad  y = 0,\\
\lim_{y\rightarrow+\infty}  u & = & U(t,x),   \\  
\end{array}
\right.
\end{equation}
that was proposed by Ludwig Prandtl \cite{Pra:1904}  in 1904 as a model for fluids with low viscosity near a solid boundary. This model is obtained formally as a singular limit  of the Navier-Stokes  equations, in the limit of vanishing viscosity. In this asymptotics, $y = 0$ is the boundary, $x$ is a curvilinear coordinate along the boundary, whereas $u=u(t,x,y)$ and $v=v(t,x,y)$ are the tangential and normal components of the velocity in the so-called boundary layer. The pressure $P=P(t,x)$ and tangential velocity $U = U(t,x)$ are given: they describe the flow just outside the boundary layer, and satisfy the Bernoulli equation 
$$\D_t U + U\D_xU + \D_x P =0. $$
Finally,  the source term $f = f(t,x,y)$ accounts for possible additional forcings. We refer to \cite{Guy:2001} for the formal asymptotic derivation and all necessary physical background.  We shall restrict here to  two settings:  
\begin{itemize}
\item {\em the initial value problem (IVP)}:  
\begin{equation} \label{IVP}
(t,x,y) \: \mbox{ in } \:  [0,T) \times \T \times \RR_+,  \: u\vert_{t=0} = u_0(x,y).
\end{equation}
\item {\em the boundary value problem (BVP)}: 
\begin{equation} \label{BVP}
(t,x,y)  \: \mbox{ in }   \: \T \times [0,X) \times \RR_+,  \: u\vert_{x=0} = u_1(t,y).
\end{equation}
\end{itemize}

Although \eqref{prandtl}  is the cornerstone of the boundary layer theory, the well-posedness of the equation and its rigorous derivation from Navier-Stokes are far from being established. The reason is that the  boundary layer undergoes many instabilities,  the impact of which  on the relevance of the Prandtl model is not clear. One  popular instability is the so-called boundary layer separation, which is created by an adverse pressure gradient ($\pa_x P > 0$) and a loss of monotonicity in $y$ of the tangential velocity: see \cite{Guy:2001}. 

Roughly, only two frameworks have led to positive mathematical results: 
\begin{itemize}
\item the analytic framework:  under analyticity of the initial data and the Euler flow, Sammartino and Caflish \cite{SC} showed the well-posedness of the initial value problem,  and successfully justified the asymptotics  locally in time. See also \cite{Lom:2003}. Up to our knowledge it is the only setting in which the Prandtl model is fully justified. 
\item the monotonic framework: under a main assumption of monotonicity in $y$ of  the initial data, Oleinik and Samokhin \cite{Ole} proved the local  well-posedness of both  the boundary value and initial value problems.   The latter result was extended to be global in time by a work of Xin and Zhang \cite{Xin} when $f = 0$ and  $\pa_x P \le 0$. 
\end{itemize} 
We refer to the review article \cite{E:2000} for precise statements and ideas of proofs. We stress that in all the aforementioned works, either analyticity or monotonicity of the initial data is assumed.  
When such assumptions are no longer satisfied, instabilities develop, and the Prandtl model is unlikely to be valid, at least globally in space time.  For instance, it was shown by E and Engquist \cite{E:1997} that $C^\infty$ solutions of \eqref{prandtl} do not always exist globally in time.  As regards the asymptotic derivation of Prandtl,  some counterexamples due to Grenier \cite{Gre:2000} have shown that the asymptotics does not hold  in the Sobolev space 
$W^{1,\infty}$.
Finally, negative results  have culminated in the recent paper \cite{GD} by the first author and Dormy, that establishes some linear ill-posedness for the initial value problem in a Sobolev setting. We shall come back to this article in due course. Broadly, the authors consider the linearized Prandtl equation around a non-monotonic  shear flow, and construct $\cO(k^{-\infty})$ approximate solutions  that grow like $e^{\sqrt k t}$ for high frequencies $k$ in $x$. 
{\em Our aim in this note is to ponder on this construction  to establish further ill-posedness results for the Prandtl equation.}  

\medskip
Let us now present our results. We only treat the case of constant   $U \ge 0$, and  shall consider perturbations of some steady shear flow solutions: 
$$(u,v) = (u_s(y), 0), \quad u_s(0) = 0, \quad \lim_{y\rightarrow +\infty} u_s = U.$$
Note that only non-trivial source terms $f$ yield non-trivial solutions of this form. But  up to minor changes, our results adapt to some  unsteady shear flows  $(u(t,y), 0)$ satisfying $\pa_t u - \pa^2_y u = 0$. Thus, the important special case $f=0$ can be treated as well.    

The system satisfied by the perturbation $(u,v)$ of $(u_s,0)$ reads:
\begin{equation}\label{prandtl2}
\left\{
\begin{array}{rlll} 
\D_t u  +  u_s \D_x u +  u_s' \, v  + u \D_x u + v \D_y u - \D^2_y  u   &=& 0,  \quad y > 0,\\
\D_x  u + \D_y v & = & 0,  \quad  y > 0, \\
u = v  &=& 0, \quad  y = 0, 
\end{array}
\right.
\end{equation}
plus the condition $\lim_{y\rightarrow+\infty}  u  =  0$, that will be encoded in the functional spaces. 

Our first result is related to the linearized version of this system, that is:   
\begin{equation}\label{prandtl3}
\left\{
\begin{array}{rlll} 
\D_t u  +  u_s \D_x u +  u_s' \, v   - \D^2_y  u   &=& 0,  \quad y > 0,\\
\D_x  u + \D_y v & = & 0,  \quad  y > 0, \\
u = v  &=& 0, \quad  y = 0.   
\end{array}
\right.
\end{equation}

We state a strong ill-posedness result for the initial value problem, namely: 
\begin{theorem}[Non-existence of solutions for linearized Prandtl] \label{theo1}
There exists a shear flow $u_s$  with $\: u_s-U  \in C^\infty_c(\RR_+)$ such that : for all $T > 0$, there exists an initial data $u_0$ satisfying  

\smallskip
\noindent
i)  $\displaystyle e^{y} u_0 \in H^\infty(\T \times \RR_+)$. 

\smallskip
\noindent
ii)  The IVP \eqref{prandtl3}-\eqref{IVP}  has no distributional solution $u$ with
$$u \in L^\infty(]0,T[; L^2(\T \times \RR_+)), \quad \pa_y u \in  L^2(]0,T[ \times \T \times  \RR_+).$$ 
\end{theorem}
We quote  that  a solution $u$ of \eqref{prandtl3}  with the above regularity satisfies 
$$ \pa_t u \in L^2(]0,T[; H^{-1}(\T \times \RR_+)),  $$
so that  in turn, 
$$ u  \in C([0,T], H^{-1}(\T \times \RR_+)) \cap C_w([0,T]; L^2(\T \times \RR_+)). $$
This gives a meaning to the initial condition.  

Theorem \ref{theo1} shows that the linearized Prandtl system \eqref{prandtl3} is ill-posed in any reasonable sense, which strengthens the result of \cite{GD}. A  difficult  open problem  is the extension of such theorem to the nonlinear setting \eqref{prandtl2}. For the time being,  we are unable to disprove the existence of a flow for the nonlinear equation. In short, we are only able to show that if the flow exists, it is not Lipschitz continuous  from $H^m$ to $H^1$, for any $m$. More precisely, we introduce the 
\begin{definition}[Lipschitz well-posedness  for Prandtl equation near a shear flow]\label{def-wellposed} 
For any smooth $u_s$ and $m \ge 0$, we say the IVP \eqref{prandtl2}--\eqref{IVP} is  locally $(H^{m},H^1)$ Lipschitz well-posed  if there are  constants $C,\delta_0,T$,  and a subspace $\mathcal{X}$ of $L^\infty(]0,T[; H^1(\T \times \RR_+))$  s.t.

\smallskip 
\noindent  
for any initial data $ u_0$ with $\displaystyle e^y u_0 \in 
H^{m}(\T \times \RR_+)$ and  $\displaystyle  \|e^y u_0 \|_{H^{m}} \le \delta_0,$
there exists a unique distributional solution $u$ of \eqref{prandtl2}-\eqref{IVP} in $\mathcal{X}$,  and there holds
\begin{equation}\label{stab-ineq} 
\mbox{ess}\sup_{0\le t \le T}\| u(t) \|_{H^1_{x,y}} 
\: \le \:  C \| e^y u_0 \|_{H^{m}_{x,y}} 
\end{equation}
\end{definition}
Let us quote that if $u$ belongs to $\mathcal{X}$, all terms in equation \eqref{prandtl3} are well-defined as distributions, including the nonlinear term 
$$u \pa_x u + v \pa_y u  \: =  \: \pa_x(u^2) + \pa_y(v \, u) \:  \in L^\infty(0,T; W^{-1,p}_{loc}(\T \times \RR_+)), \quad \forall p < 2. $$
Again, the solution $u$ satisfies 
$$ \pa_t u \in  L^\infty(0,T; W^{-1,p}_{loc}(\T \times \RR_+)), \quad u \in    C([0,T]; W^{-1,p}_{loc}(\T \times \RR_+)), \quad \forall p < 2,$$
which gives a meaning to the initial condition.  We can now state 
\begin{theorem}[No Lipschitz continuity of the flow] \label{theo2}
There exists a shear flow $u_s$ with $u_s - U  \in C^\infty_c(\RR_+)$ such that: for all $m \ge 0$, the Cauchy problem \eqref{prandtl2}-\eqref{IVP} is not  locally $(H^m,H^1)$ Lipschitz well-posed. 
 \end{theorem}

The proof of the linear ill-posedness result, Theorem \ref{theo1}, is based on the previous construction (\cite{GD}) of a strong unstable quasimode for  the linearized Prandtl operator, together with a use of the standard closed graph theorem. For Theorem \ref{theo2}, we make a simple use of an idea of Guo and Tice (\cite{GT}) on deriving the ill-posedness of the flow from a strong linearized instability result.

\medskip
We shall end this introduction with
 space instability results, related to the boundary value problem. Let us stress that space instability is very natural in boundary layer theory. Indeed,  many works on boundary layers are related to steady problems for flows around obstacles. In this context,   $x$ can be seen as the evolution variable, $x=0$ corresponding to the leading edge of the obstacle. Boundary layer separation, that takes place upstream from the leading edge, can also be seen  as a blow up phenomenon in space.  One can also refer to the paper by M\'etivier \cite{Met}, for a mathematical study  of spacial instabilities in the context of Zakharov equations.

Precisely, at the linear level, we prove the following:

 \begin{theorem}[Spacial ill-posedness for linearized Prandtl equation] \label{theo3}
Let $U > 0$. There exists a shear flow $u_s$  with $u_s(y) > 0$ for $y > 0$,  $\: u_s-U  \in C^\infty_c(\RR_+)$ and such that : for all $X > 0$, there exists an initial data $u_1$ satisfying  

\smallskip
\noindent
i)  $\displaystyle e^{y} u_1 \in H^\infty(\T \times \RR_+)$. 

\smallskip
\noindent
ii)  The BVP  \eqref{prandtl2}-\eqref{BVP} has no weak solution $u$ with 
$$u_s \, u \in L^2_t(\T; C_x([0,X] ; H^2_y(\RR_+))), \quad  u  \in L^2_{t,x}(\T \times (0,X); H^2_y(\RR_+)).$$
\end{theorem}
This is of course an analogue of Theorem \ref{theo1}.  From there, one could also obtain an analogue of Theorem \ref{theo2} : we skip it for the sake of brevity.

The proof of these spacial results are again based on construction of an unstable quasimode for the linearized operator. It turns out that the instability mechanism introduced in \cite{GD} to construct these unstable modes can be modified in such a way that it yields the ill-posedness. One should note that in course of deriving the non-existence result, we are obliged to prove a {\em uniqueness} result, and as it turns out, obtaining such a result is not as straightforward as in the case for the IVP problem.  The difficulty lies in the fact that we now view the equation as an evolution in $x$ and it is not at all obvious for one to obtain certain {\em energy or a priori estimates} for solutions of the BVP problem. Nevertheless, we present a proof of the uniqueness result in the last section of the paper.

\medskip
The outline of the paper is as follows: section \ref{sec1} details the non-existence result stated in Theorem \ref{theo1}. The nonlinear result is explained in section \ref{sec2}. We show in section \ref{sec3} how to adapt the arguments of \cite{GD} to get spacial ill-posedness. 
We shall conclude the paper with
 some comments on this spacial instability, and how it relates to the well-known results of Oleinik.

 \section{The linearized IVP} \label{sec1}
This section is devoted to the proof of Theorem \ref{theo1}.  We start by recalling some key elements of article \cite{GD}. It  deals with the ill-posedness of the linearized Prandtl system \eqref{prandtl3} in  the Sobolev setting. The  high  frequencies in $x$  are investigated.  The main point in the article is the construction of a strong unstable quasimode for  the linearized Prandtl operator $\cL_s u: = u_s\D_x u + v\D_y u_s - \D^2_y u$. 
It is achieved under the main assumption that $u_s$ has a non-degenerate critical point $a > 0$: 
$\displaystyle  u_s'(a) = 0$, $\: \displaystyle u''_s(a) < 0$. 
More precisely, if $u_s$ takes the form 
\begin{equation} \label{specialus}
u_s \: = \:  u_s(a) - \frac{1}{2}(y-a)^2 \quad \mbox{in the vicinity of }  \: a, 
\end{equation}
 one can build accurate approximate solutions  $(u^n_\eps,v^n_\eps)$  of \eqref{prandtl2} {\em that have $x$-frequency  $\eps^{-1}$ and grow exponentially at rate $\eps^{-1/2}$}.  Namely, 
 $$ u^n_\eps(t,x,y) \: = \: i \, e^{i \eps^{-1} x} \, e^{i \eps^{-1} \omega(\eps) t} \, U^n_\eps(y), \quad v^n_\eps(t,x,y)   \: = \: \eps^{-1} \,  e^{i \eps^{-1} x} \, e^{i \eps^{-1} \omega(\eps) t} \, V^n_\eps(y) $$ 
 where 
 $$ \omega(\eps) \: = \: - u_s(a) \: + \: \eps^{1/2} \tau, \quad \mbox{ for some } \: \tau \: \mbox{ with } \: \I \tau < 0, $$
and $U^n_\eps$, $V^n_\eps$ are smooth functions of $y$. These functions are expansions of boundary layer type, made of $\cO(n)$ terms,  with both a regular and a singular part in $\eps$. In particular, one has
$$ \| U^n_\eps \|_{L^2(\RR_+)} \: \ge \: c_n, \quad  \| e^y \, U^n_\eps \|_{H^k(\RR_+)} \: \le \:  C_{n,k} \, ( 1 \: + \:  \eps^{-(k-1)/4} ), \quad \forall k \in \NN  $$
the loss in  $\eps$  being due to the singular part.  By accurate approximate solutions, we mean that 
$\displaystyle  \pa_t u_\eps^n +  \cL_s u_\eps^n \: = \: r_\eps^n, $ with
$$ r^n_\eps(t,x,y) \: = \: e^{i \eps^{-1} x} \, e^{i \eps^{-1} \omega(\eps) t} \, R_\eps^n, \quad \| e^y R_\eps^n \|_{H^k(\RR_+)} \le C_{n,k} \, \eps^n, \quad \forall k. $$
Again, we refer to \cite{GD} for all necessary details. Actually, the construction of \cite{GD}, that deals with a time dependent $u_s$, can be much simplified in the  case of our steady flow $u_s$.  A similar construction will be described at the end of the paper, when dealing with spacial instability. 

\medskip
We can now turn to the proof of the theorem. It ponders on the previous construction and on the use of the closed graph theorem. We refer to Lax \cite{Lax} for similar arguments in the context of geometric optics.  We  argue by contradiction. Let us assume that for some $T > 0$, and for any data $u_0$ with $e^y u_0 \in H^\infty(\T \times \RR_+)$, there is a unique solution of \eqref{prandtl3} 
$$u \in L^\infty(]0,T[; L^2(\T \times \RR_+)), \quad \pa_y u \in  L^2(]0,T[ \times \T \times  \RR_+).$$ 
We can then define  
$$ \mathcal{T} \: : \:  e^{-y} H^\infty(\T \times \RR_+) \:  \mapsto \:  L^\infty(]0,T[; L^2(\TT \times \RR_+)) \times  L^2(]0,T[ \times \T \times  \RR_+) , \quad  
u_0 \mapsto (u, \pa_y u).  $$ 
It is a linear map between Fr\'echet spaces, and it is easy to check that it has closed graph. By the closed graph theorem, we deduce that $\mathcal{T}$ is bounded. This means that for some $K \in \NN$, for all $ u_0 \in e^{-y} H^\infty(\T \times \RR_+)$, 
$$ \sup_{t \in [0,T]} \| u(t) \|_{L^2_{x,y}} \: + \: \| \pa_y u \|_{L^2_{t,x,y}} \: \le \: \mathcal{C} \, \| e^{y} u_0 \|_{H^K_{x,y}}.  $$ 
We quote that this bounds hold for the supremum in time and not only for the essential supremum,  thanks to the weak continuity in time  of $u$ with values in $L^2$. 

Let us define the family of linear operators 
$$ S(t) \: : \:   e^{-y} H^\infty(\T \times \RR_+) \:  \mapsto  \:  L^2(\TT \times \RR_+)), \quad u_0 \: \mapsto \: \left(\mathcal{T}u_0\right)_1(s), \quad t \in [0,T], $$
that is $S(t)u_0 = u(t)$, where $u$ is the unique solution of \eqref{prandtl3} with initial data $u_0$. 
The previous inequality shows that $S(t)$ extends into a bounded linear operator from $e^{-y} H^K$ to $L^2$, with a bound $\mathcal{C}$ independent of $t$. 

We now  introduce $u(t) \: := \:  S(t) u_\eps^n(0)$, and $\: v \: = \: u - u_\eps^n$, where $u^n_\eps$ is the growing function defined above. It satisfies 
\begin{equation} \label{equationv}
 \pa_t   v +  \cL_s  v = - r_\eps^n, \quad v\vert_{t=0} = 0. 
 \end{equation} 
We claim that $v$ has the Duhamel representation
$$ v(t) \: = \: - \int_0^t S(t-s) r^n_\eps(s) \, ds $$ 
Indeed, as $r^n_\eps$ is continuous with values in  $e^{-y} H^\infty(\T \times \RR_+)$, the integral is well-defined, and  straightforward differentiation with respect to $t$ shows that it defines another solution  $\tilde v$ of  \eqref{equationv}. 
Thus, the difference $ w = v - \tilde v$ satisfies $ \pa_t   w +  \cL_s w = 0$, $\: w\vert_{t=0} = 0$, with regularity 
$$w \in L^\infty(]0,T[; L^2(\T \times \RR_+)), \quad \pa_y w \in  L^2(0,T \times \T \times  \RR_+).$$ 
By our uniqueness  assumption for this equation, we obtain $w=0$, that is the Duhamel formula.

On one hand, thanks to the bound on the $S(t)$ and the remainder $r^n_\eps$, we get that 
$$ \| u(t)  \| _{L^2_{x,y}} \: \le \: \mathcal{C} \| e^y \, u^n_\eps(0) \|_{H^K_{x,y}} \: \le \: C_K \, \eps^{-K},   $$
 as well as 
$$ \| v (t) \|_{L^2_{x,y}} \: \le \: \mathcal{C} \, \int_0^t \| e^{y} r^n_\eps(s) \|_{H_{x,y}^K}(s) ds \: \le \:  \,  C_{n,K} \, \eps^{n+\frac{1}{2}-K} \,    
e^{\frac{|\Im \tau| t}{\sqrt{\eps}}}, \quad \forall \, n,K \ge 0.  $$
On the other hand, one has:  
$$ \| u^n_\eps(t) \|_{L^2_{x,y}} \: \ge \:  c_n \, e^{\frac{|\Im \tau| t}{\sqrt{\eps}}}$$ 
by the properties of $u_\eps^n$ recalled below. Combining the last three inequalities, we deduce
$$   C_K \, \eps^{-K} \: \ge \:   \| u(t)  \| _{L^2_{x,y}} \: \ge \:  \| u^n_\eps(t) \|_{L^2_{x,y}} - \| v (t) \|_{L^2_{x,y}}  \: \ge \: \left( c_n  \, - \,  C_{n,K} \,
 \eps^{n+\frac{1}{2}-K} \right)  \, e^{\frac{|\Im \tau| t}{\sqrt{\eps}}}. $$
This yields a contradiction for small enough $\eps$,  if we take $n$ and $t$ such that   
$$n+\frac{1}{2} \:  > \:  K,  \quad t \: > \: \frac{(\ln(C_{n,K}/c_n) + K | \ln\eps | ) \sqrt{\eps}}{|\Im \tau|}.$$    
 So far, we have proved that under  assumption \eqref{specialus},  there is   for any  $T > 0$ an initial data $u_0 \in e^{y} H^\infty(\T \times \RR_+)$ for which  either existence or uniqueness of a solution of \eqref{prandtl3} on $[0,T]$ fails.  To rule out a possible  lack of uniqueness, we further assume that 
 \begin{equation} \label{specialus2}
 \sup_{t\ge 0}\Big(\sup_{y\ge 0}|u_s|+\int_0^\infty y|\D_y u_s|^2dy\Big)< +\infty.
 \end{equation}
 Theorem \ref{theo1} is then a consequence of the following uniqueness result: 
\begin{proposition}\label{theo-uniquelin} 
Let $u_s$ be a smooth shear flow  satisfying \eqref{specialus2}. Let 
$$ w \in L^\infty(]0,T[; L^2_x(\T \times \RR_+)), \quad \pa_y w \in  L^2(0,T \times \T \times  \RR_+).$$ 
a solution of \eqref{prandtl3} with $w\vert_{t=0} = 0$. Then, $w \equiv 0$. 
\end{proposition}
\begin{proof} Let us define   $\hat w_k(t,y)$, $k\in \ZZ$, the Fourier transform of $w(t,x,y)$ in $x$ variable. We observe that for each $k$, $\hat w_k$ solves
\begin{equation}\label{eqs-wk}\left\{\begin{array}{rll}\D_t \hat w_k + ik u_s \hat w_k - ik \D_yu_s \int_0^y \hat w_k (y')dy' - \D_y^2 \hat w_k &=&0\\\hat w_k(t,0) &=& 0\\\hat w_k(0,y)&=&0.\end{array}\right.\end{equation}

Taking the standard inner product of the equation \eqref{eqs-wk} against the complex conjugate of $\hat w_k$ and using the standard Cauchy--Schwartz inequality to the term $\int_0^y\hat w_kdy'$, we obtain
$$\begin{aligned}\frac 12 \dt\|\hat w_k\|^2_{L^2(\RR_+)} + \|\D_y \hat w_k\|^2_{L^2(\RR_+)} &\le |k|\int_0^\infty|u_s||\hat w_k|^2 dy + |k|\int_0^\infty|\D_yu_s|y^{1/2}|\hat w_k| \|\hat w_k\|_{L^2(\RR_+)}dy\\&\le |k|\Big(\sup_{t,y}|u_s|+\int_0^\infty y|\D_yu_s|^2dy\Big)\|\hat w_k\|^2_{L^2(\RR_+)}.\end{aligned}$$
Applying the Gronwall lemma into the last inequality yields $$
\|\hat w_k(t)\|_{L^2(\RR_+)} \le C e^{C|k|t} \|\hat w_k(0)\|_{L^2(\RR_+)},$$ for some constant $C$.
Thus, $\hat w_k(t)\equiv 0$ for each $k\in \ZZ$ since $\hat w_k(0)\equiv 0$. That is, $w\equiv 0$, and the theorem is proved.
\end{proof}

\section{The nonlinear IVP} \label{sec2}
With Theorem \ref{theo1} at hand, we can turn to the nonlinear statement of Theorem \ref{theo2}. We shall make use of  an idea of Guo and Tice \cite{GT} on deriving a nonlinear instability  from a strong linearized one. Note that the nonlinear equation in \eqref{prandtl2} reads 
$\pa_t u + \cL_s u = N(u),$ with  $\cL_s$ as in the previous section, and nonlinearity
 $N(u):=-u\D_xu - v\D_yu$.

\medskip
We now prove the theorem by contradiction. Let $u_s$ be as in Theorem \ref{theo1}. Assume that the IVP \eqref{prandtl2}-\eqref{IVP} is $(H^{m}, H^1)$ locally Lipschitz well-posed for some $m \ge 0$. Let $C, \delta_0, T$ be the constants given in the definition of Lipschitz 
well-posedness. By Theorem~\ref{theo1}, there is an initial data $u_0 \in e^{-y} H^\infty(\T \times \RR_+)$ that does not generate any solution $u$ of the linearized equation \eqref{prandtl3} with 
$$u \in L^\infty(]0,T[; L^2_x(\T \times \RR_+)), \quad \pa_y u \in  L^2(]0,T[ \times \T \times  \RR_+).$$  
Up to multiplication by a constant, we can assume  $\| e^y u_0 \|_{H^m} \: = \: 1$. Let us take $v_0^\delta \: := \:  \delta u_0$, with $\delta$ a small parameter less than $\delta_0$. By the Lipschitz well-posedness hypothesis, there is a  solution $v^\delta$ of \eqref{prandtl2} on $[0,T]$ with initial data $v_0^\delta$. Moreover, 
$$ \mbox{ess}\sup_{t \in [0,T]}\| v^\delta(t) \|_{H^1_{x,y}} \: \le \: C \, \delta.$$
 In other words, $u^\delta = v^\delta/\delta$ is  bounded in $L^\infty(0,T; H^1)$  uniformly with respect to $\delta$, and moreover 
\begin{equation}\label{eqs-w}
\D_t u^{\delta} + \cL_s u^{\delta} = \delta N(u^{\delta}), \qquad u^{\delta}(0,x,y) = u_0. 
\end{equation}
From the  bound on $u^\delta$, we deduce that, up to  a subsequence, 
$$ u^\delta \rightarrow u \quad  L^\infty(0,T; H^1(\T \times \RR_+)) \mbox{ weak *  as } \delta \rightarrow 0. $$ 
Furthermore, the nonlinearity $\delta N(u^\delta)$ goes to zero strongly in $L^\infty(0,T; W^{-1,p }_{loc})$, for all $p < 2$.  We end up with 
$$ \pa_t u + \cL_s u \: = \:  0, \quad u\vert_{t=0} = u_0.  $$
As 
$$u \in L^\infty(]0,T[; L^2(\T \times \RR_+)), \quad \pa_y u \in  L^2(]0,T[ \times \T \times  \RR_+).$$ 
this contradicts the result of  non-existence of solutions starting from $u_0$.

\section{Spacial instability}\label{sec3}
This section is devoted to the boundary value problem for the linearized Prandtl equation. 
We assume $U > 0$, and consider some shear flow   $u_s $ with  $u_s(y) > 0$ for $ \: y > 0$. As before,  we assume $u_s - U \in C^\infty_c(\RR_+)$, and 
$$ u_s \: = \: u_s(a) + u''_s(a) \frac{(y-a)^2}{2} \: \mbox{ in the vicinity  of some } \:  a > 0, \quad  \mbox{ with } \: u''_s(a) < 0.  $$ 
As the time $t$ and the space $x$ are somewhat symmetric in the Prandtl equation, we may adapt the construction of the unstable quasimode performed in \cite{GD}. We sketch this construction in the next paragraph, and  then turn to the proof of Theorem \ref{theo3}.

\subsection{The unstable quasimode} 
The aim of this paragraph is  to construct an approximate solution of \eqref{prandtl3}, that has high time frequency $\eps^{-1}$ and grows exponentially for positive $x$ at rate $\sqrt{\eps}^{-1}$. 
We look for growing solutions in the form
\begin{equation} \label{Fourier}
u(t,x,y) = e^{-it/\e - i\w(\e) x/\e}  u_\e(y),\qquad v(t,x,y) = \e^{-1} e^{-it/\e-i\w(\e)x/\e}  v_\e(y), \qquad \e>0.
\end{equation}
We plug the Ansatz into \eqref{prandtl3},  and eliminate $u_\e$ by the divergence free  condition $u_\e = - \frac{i v_\e'}{\w(\e)}$. We end up with
\begin{equation}\label{spectral-sys}\left\{\begin{array}{rll} (1+\w(\e)u_s)v_\e' - \w(\e) u_s' v_\e - i\e v_\e^{(3)}=0,& \quad y>0\\ {v_\e}_{|y=0}={v_\e'}_{|y=0}=0.&\end{array}\right.\end{equation}
Introducing  the notations
$$\tilde \omega(\eps) \: := \: \omega(\eps)^{-1}, \quad  \tilde \eps \: := \:   - \omega(\eps)^{-1} \eps,$$
it reads
\begin{equation}\label{spectral-sys2}\left\{\begin{array}{rll} 
(\tilde \w(\e) + u_s)v_\e' - u_s'v_\e + i \tilde\e v_\e^{(3)}=0,& \quad y>0\\ {v_\e}_{|y=0}={v_{ \e}'}_{|y=0}=0.&\end{array}\right.\end{equation}
Thus,  the equation gets formally close to the equation
\begin{equation}\label{spectral-sys3}\left\{\begin{array}{rll} 
(\w(\e) + u_s)v_\e' - u_s'v_\e + i\e v_\e^{(3)}=0,& \quad y>0\\ {v_\e}_{|y=0}={v_\e'}_{|y=0}=0,&\end{array}\right.\end{equation}
which has been studied in  \cite{GD}, in connection to the IVP for \eqref{prandtl3}. More precisely, the authors build  an approximate solution of \eqref{spectral-sys3} under the form
\begin{equation} \label{ansatz2}
\left\{
\begin{aligned}
\omega^{app}(\eps) \: & =  \: -u_s(a) \: + \: \eps^{1/2} \tau, \\
v_\eps^{app}(y) \: &  \: =  \, H(y-a) (u_s + \omega(\eps))   \: + \:  \eps^{1/2}  V\left( \frac{y-a}{\eps^{1/4}} \right). 
\end{aligned}
\right.
\end{equation} 
 The streamfunction $v^{app}_{\eps}$ divides into two parts: a regular part 
$$\displaystyle v_\eps^{reg}(y) \: := \: H(y-a) (u_s + \omega(\eps)),$$
 and a "shear layer part" 
$\displaystyle v_\eps^{sl}(y) \: := \: \eps^{1/2}  V\left( \frac{y-a}{\eps^{1/4}} \right)$. 
Note that $(\omega^{app}(\eps), v_\eps^{reg})$  solves \eqref{prandtl3} except for the $\cO(\eps)$ term coming from  diffusion. However, both $v_\eps^{reg}$ and its second derivative have discontinuities at $y=a$. This  explains the introduction of $v_\eps^{sl}$. The shear layer profile $V = V(z)$ satisfies the system 
\begin{equation} \label{SL}
\left\{
\begin{aligned}
& \left(\tau + u_s''(a) \frac{z^2}{2}\right) V' \: - \:  u_s''(a)  \, z \, V \: +
\: i  \,  V^{(3)} \: = \: 0, \quad  \: z \neq 0, \\ 
& \left[ V \right]_{\vert_{z=0}} \: = \:  - \tau, \quad  \left[ V' \right]_{\vert_{z=0}} \: =
\: 0, \quad  \left[ V'' \right]_{\vert_{z=0}} = - u''(a), \\
& \lim_{\pm\infty} V \: = \: 0. 
\end{aligned}
\right.
\end{equation}
The jump conditions on $V$ compensate those of the regular part. The parameter $\tau$ allows the system not to be overdetermined. Indeed, one shows  that for some appropriate $\tau$ with $\Im \tau < 0$, \eqref{SL} has a solution, with rapid decay to zero as $z \rightarrow \pm \infty$. This singular perturbation is responsible for the instability, through the eigenvalue perturbation $\tau$. We refer to \cite{GD} for all necessary details.  
Note that writing 
 $$ \tilde V(z) \: = \: V(z) \: + \: {\bf 1}_{\RR_+} \left(  \tau + u_s''(a)
\frac{z^2}{2} \right),$$
one  gets rid of the jump conditions: 
\begin{equation} 
\left\{
\begin{aligned} \label{eqtildeV}
& \left(\tau + u_s''(a) \frac{z^2}{2}\right)  \tilde V' \: - \:  u_s''(a)
  \, z \,  \tilde V \: + 
\: i  \,  \tilde V^{(3)} \: = \: 0, \quad z \in \RR, \\ 
& \lim_{z \rightarrow -\infty}  \tilde V \: = \: 0, \qquad 
  \tilde V \: \sim \: \tau + u_s''(a) \frac{z^2}{2} \: \mbox{ as } \: z \rightarrow +\infty 
\end{aligned}
\right.
\end{equation}
Accordingly, the approximate solution \eqref{ansatz2}  takes the slightly simpler form
\begin{equation} \label{ansatz3}
\left\{
\begin{aligned}
\omega^{app}(\eps) \: & =  \: -u_s(a) \: + \: \eps^{1/2} \tau, \\
v_\eps^{app}(y) \: &  \: =  \, H(y-a) \left(u_s - u_s(a) -  u''_s(a) \frac{(y-a)^2}{2}\right)  \: + \:  \eps^{1/2}  \tilde V\left( \frac{y-a}{\eps^{1/4}} \right). 
\end{aligned}
\right.
\end{equation}

\medskip
On the basis of this former  analysis, it is tempting to consider for our system \eqref{spectral-sys2} the following expansion:
 \begin{equation} \label{ansatz4}
\left\{
\begin{aligned}
\tilde \omega^{app}(\eps) \: & =  \: -u_s(a) \: + \: \tilde\eps^{1/2} \tau, \\
v_\eps^{app}(y) \: &  \: =  \, H(y-a) \left(u_s - u_s(a) - u''_s(a) \frac{(y-a)^2}{2}\right)  \: + \: \tilde \eps^{1/2}  \tilde V\left( \frac{y-a}{\tilde \eps^{1/4}} \right). 
\end{aligned}
\right.
\end{equation} 
However, it  is not so straightforward: 
\begin{itemize}
\item Back to the initial notations, the first relation reads 
 \begin{equation} \label{eqomegaeps}
 \omega^{app}(\eps)^{-1} \:  =  \: -u_s(a) \: + \: (-\omega^{app}(\eps)^{-1}\eps)^{1/2} \tau. 
 \end{equation} 
 In particular, it is no longer an {\em explicit} definition of $\omega^{app}(\eps)$. It is an equation for  $\omega^{app}(\eps)$, and we must check that 
 the equation defines it {\em implicitly}.
\item  Assuming that $\omega^{app}(\eps)$ is defined,  $\tilde \eps$ is some unreal complex number. Hence, $v^{app}_\eps$  is not {\it a priori}  properly defined, because $\tilde V$ 
no longer has a real argument.
\end{itemize}
Fortunately, as we will now show, these difficulties can be solved, leading to the desired unstable quasimode. 
 
 \medskip
We start with equation \eqref{eqomegaeps}. A standard application  of the implicit function theorem implies that for $\eps$ small enough,  the complex equation
$$ F(z) = 0, \quad  F(z) \: := \: z + u_s(a) - (-\eps z)^{1/2}  \tau$$
(where $z^{1/2}$ is  the principal value of the square root) has a unique solution near $-u_s(a)$. This allows to define $\omega^{app}(\eps)^{-1}$, so $\omega^{app}(\eps)$. Furthermore,  easy  computations  yield   
\begin{equation} \label{omegaepsapp}
\omega^{app}(\eps) \: = \: -\frac{1}{u_s(a)} - \eps^{1/2}  \frac{\tau}{u_s(a)^{3/2}} + o(\eps^{1/2}). 
\end{equation} 
Note that  injecting this expression in  \eqref{Fourier}  gives some exponential growth  in $x$ at rate $\eps^{-1/2}$. 

\medskip
It remains  to clarify  the definition of $v_\eps$ in \eqref{ansatz4}. Of course,  the complex powers $\tilde \eps^{1/2}$, $\tilde \eps^{1/4}$ can still be defined. The problem lies in the profile $\tilde V = \tilde V(z)$, which is  a function over $\RR$, instead of $(-\omega^{app}(\eps))^{-1/4} \RR$. 
To overcome this problem, we will show that $\tilde V$  extends  holomorphically to  the  neighborhood of the real line
$$U_\tau \: :=  \:  \C\setminus \left(  i(-\tau)^{1/2}[-\infty, -1] \cup  i(-\tau)^{1/2}[1, +\infty] \right).$$  
The extension will satisfy the ODE in \eqref{eqtildeV} over $U_\tau$,  and will have the following asymptotic behaviour, for small $\theta \ge  0$:
$$\lim_{z \in e^{i\theta} \RR,  \: z \rightarrow  -e^{i \theta} \infty} \tilde V = 0, \qquad V \: \sim \: \tau + u_s''(a) \frac{z^2}{2} \: \mbox{ as } \: z \in e^{i \theta} \RR, \:\:  z  \rightarrow +e^{i\theta} \infty.  $$ 
In particular, $\tilde V$ will satisfy all the requirements along the line  $(-\omega(\eps))^{1/4} \RR$, for $\eps$ small enough.

\medskip
To perform our extension, we  write $\tilde V$ as  
$$ \tilde V(z) \: := \:  \left(\tau + u''_s(a) \frac{z^2}{2} \right) W(z) $$
where $W$ satisfies 
\begin{equation*} 
\left\{
\begin{aligned}
& \left(\tau + u_s''(a) \, z^2/2 \right)^2   \frac{d}{dz} W \: +
\: i \, \frac{d^3}{dz^3} \left(  \left( \tau + u_s''(a)
\, z^2/2 \right) \,  W \right)  \: = \: 0, \\ 
& \lim_{-\infty}   W \: = \: 0, \quad \lim_{+\infty}
  W \: = \: 1.
\end{aligned}
\right.
\end{equation*}
Then, performing the changes of variables 
$$ \tau \: = \: \frac{1}{\sqrt{2}}|u_s''(a)|^{1/2} \, \tilde \tau, \quad z \: =
  2^{1/4} \, |u_s''(a)|^{-1/4}  \tilde z $$
leaves us with 
\begin{equation*} 
\left\{
\begin{aligned}
& \left(\tilde \tau  -  \tilde z^2 \right)^2   \frac{d}{d\tilde z} \tilde W \: +
\: i\, \frac{d^3}{d\tilde z^3} \left(  \left( \tilde \tau -  \tilde z^2 \right) \,  \tilde W \right)  \: = \: 0, \\ 
& \lim_{-\infty}   \tilde W \: = \: 0, \quad \lim_{+\infty}
  \tilde W \: = \: 1.
\end{aligned}
\right.
\end{equation*}
This system already appears  in \cite{GD}. It is shown that for some appropriate $\tilde \tau$ with negative imaginary part, there exists some solution $\tilde W$.  Actually, much more is known on $\tilde W$.  Indeed,  $X = \tilde W'$  satisfies the ODE with holomorphic coefficients 
$$i (\tilde \tau - \tilde z^2)  \, X'' \: - \:   6 i \, \tilde z  \, X'' \: + \:  (\tau - \tilde z^2)^2 \: - \:  6 i \,  X = 0, $$
In particular, the coefficient of the leading order term does not vanish in $U_\tau$. 
Thus, $X$ can be extended into a holomorphic solution in $U_\tau$.  
Moreover, the previous equation on $X$ can be put into the  form of a first order system:
\begin{equation} \label{firstorder}
\frac{d}{d\tilde z} \mathcal{X} \: = \:  \tilde z \mathcal{B}( \tilde z) \mathcal{X}, \quad 
 \mathcal{X} \: = \: \left( \begin{smallmatrix}  X \\ \tilde z^{-1}
  \frac{d}{d\tilde z} X
 \end{smallmatrix} \right), 
  \quad  \mathcal{B}(\tilde z) \: = \: \left( \begin{smallmatrix} 0 & 1 \\ \frac{6 +
      i (\tau - \tilde z^2)^2}{ \tilde z^2 (\tau - \tilde z^2)}  & \frac{6}{\tau - \tilde z^2} -
    \frac{1}{\tilde z^2} \end{smallmatrix} \right).
\end{equation}
In particular,  $ \mathcal{B}$ is holomorphic at infinity, with
$\displaystyle   \mathcal{B}(\infty) \: = \: \left( \begin{smallmatrix} 0 & 1 \\ -i  &
  0 \end{smallmatrix} \right)$. It has the two distinct eigenvalues $\pm i e^{i \pi/4}$.  From there, it is possible to obtain explicit  asymptotic expansions for $X$ at infinity. Following  \cite[Theorem 5.1 p163]{Cod:1955},  in any closed sector  $\displaystyle \overline{S}_{\theta,\theta'}$ 
     inside which $\displaystyle \R (i e^{i \pi/4} z^2)$ does not cancel, there are  solutions  $\mathcal{X}_\pm$ (depending {\it a priori}
on $\theta,\theta')$  with the following asymptotic behaviour:  
\begin{equation} 
  \mathcal{X}_{\pm}   \: \sim  \:  \left(  \sum_{i\ge 0}   \mathcal{X}^i_\pm \, \tilde z^{\alpha_{\pm}}  \right) \, 
e^{P_{\pm}(\tilde z)}, \:  \quad |\tilde z| \rightarrow  +\infty, \quad \tilde z \in \overline{S}_{\theta,\theta'},  
\end{equation}
where $\alpha_{\pm}$ is a complex constant, and  
$P_{\pm}(\tilde z)$ is a polynomial of degree 2. Moreover, the  leading term of
$\displaystyle P_{\pm}$ is $\pm \frac{i e^{i\pi/4} \tilde z^2}{2}$. In the present case, tedious computations yield 
$$  \mathcal{X}_\pm(\tilde z)  \: \sim \:  C_\pm \, \tilde z^{\alpha_\pm}  \, \exp\left(\pm i e^{i\pi/4} \frac{\tilde z^2}{2}\right), \quad   \mbox{ for some } \: C_+, \alpha_+ \in \CC, $$
as $|\tilde z| \rightarrow + \infty$, in any  closed sector that is strictly inside $S_{-\frac{\pi}{8}+ \frac{k\pi}{2},\frac{3 \pi}{8}+ \frac{k\pi}{2}}$ for some $k \in \ZZ$.  As $X$ decays to zero in 
$\pm \infty$, we deduce that: for all $\delta > 0$, 
$$ |X(\tilde z)| \: \le \:  C \exp(-\alpha |\tilde z|^2), \mbox{ for some } \alpha  > 0, \quad z \in \overline{S}_{-\frac{\pi}{8}+\delta,\frac{3 \pi}{8}-\delta} \cup \overline{S}_{\frac{7\pi}{8}+\delta,\frac{11 \pi}{8}-\delta}.$$
This implies easily  that the antiderivative $\tilde W$ (now defined in $U_\tau$) satisfies 
$$ |\tilde W(\tilde z) - 1| \le C \exp(-\alpha' |\tilde z|^2),  \mbox{ uniformly in } \: \overline{S}_{-\frac{\pi}{8}+\delta,\frac{3 \pi}{8}-\delta} ,  $$
and similarly
$$ |\tilde W(\tilde z)| \le C \exp(-\alpha' |\tilde z|^2),  \mbox{ uniformly in } \:   \overline{S}_{\frac{7\pi}{8}+\delta,\frac{11 \pi}{8}-\delta}. $$
From there, one can go back to  $W$, and then to $\tilde V$. 

\medskip
In this way, we obtain an appropriate approximate solution of 
\eqref{spectral-sys2}.
 Exactly  as in \cite{GD}, we can refine this approximation by the addition of higher order "regular" terms in  $v^{app}_\eps$. Namely, we consider an expansion of the form
\begin{equation*}
 v_\eps^{app}(y)   \: =  \, H(y-a) \left(u_s - u_s(a) - u''_s(a) \frac{(y-a)^2}{2}\right)  \: + \: \tilde \eps^{1/2}  \tilde V\left( \frac{y-a}{\tilde \eps^{1/4}} \right) 
  +  \sum_{i=2}^n \tilde{\eps}^{i/2} v_{i,reg}(y) 
 \end{equation*}
Each additional term  $v_{i,reg}$  satisfies a first order equation of the type 
 $$   (\tilde \w(\e) + u_s) v'_{i,reg} - u_s' v_{i,reg} \: =  \: f^i,  $$
where $f^i$ comes from lower order terms. For example, $f^2 \: := \: H(y-a) \pa^3_y u_s$. As explained in \cite{GD}, due to the quadratic structure of $u_s$ near $a$, one can show by an easy induction that $f^i$ is identically zero near $a$, so that 
$$ v_{i,reg} \:  = \: H(y-a) \Bigl( u_s(y) + \tilde \omega(\eps) \Bigr)
\int_{a}^y \frac{f^i}{\left(u_s(z) + \tilde \omega(\eps)\right)^2} \, dz,$$  
is smooth across $y = a$. 

Back to \eqref{Fourier}, we obtain some approximate solution of \eqref{prandtl3}, with frequency $\eps^{-1}$ in $t$, that grows exponentially with $x > 0$ at rate $ \eps^{-1/2}\sigma$, 
$$\sigma \, :=  \, \frac{|\I(\tau)|}{ u_s(a)^{3/2}}.$$
This approximation solves \eqref{prandtl3} up to some $\cO\left(\eps^{n-k}  e^{\sigma x/\eps^{1/2}}\right)$ error term in  $e^{-y} H^k$ for all $k$. 
 Again, we refer to \cite {GD} for all details.

\subsection{Proof of spacial ill-posedness}
Thanks to the spacially unstable quasimode  from  the last paragraph, one  can try to  mimic  the proof of Theorem \ref{theo1}, to derive Theorem \ref{theo3}.  Indeed,  all arguments related to  the closed graph theorem adapt straightforwardly, as well as those disproving   the  continuity of the flow if it exists. In other words, we leave to the reader to check the following fact: for any $X > 0$, there exists some boundary data $\displaystyle u_1 \in e^{-y} H^\infty(\T \times \RR_+)$, for which either existence or uniqueness of a solution $u$  with 
$$u_s \, u \in L^2_t(\T; C_x([0,X] ; H^2_y(\RR_+))), \quad  u  \in L^2_{t,x}(\T \times (0,X); H^2_y(\RR_+)).$$
fails.  

\medskip
To complete the proof of the Theorem, it remains to show some uniqueness result for the BVP in this functional setting. It turns out to be more difficult than for the IVP.  Therefore, we  strengthen the assumptions on  $u_s$: besides all assumptions needed for the construction of the quasimode, we make the following hypothesis: the function $u_s$, which belongs to  $U + C^\infty_c(\RR_+)$, satisfies 
$$ u_s(y)  = y, \quad \mbox{ in the vicinity  of } \: y=0. $$
 Under this additional hypothesis, we are able to state
 \begin{proposition}[Uniqueness for the BVP] \label{uniqueBVP}
 Let $u$ be a weak
   solution of \eqref{prandtl3} such that 
 $$u_s \, u \in L^2_t(\T; C_x([0,X] ; H^2_y(\RR_+))), \quad  u  \in L^2_{t,x}(\T \times (0,X); H^2_y(\RR_+)), $$
and $u\vert_{x=0} = 0$. Then $u \equiv 0$. 
\end{proposition}
\begin{proof}
Up to a replacement of
$u$ by $P_N u$, where $P_N$   is the projector on  temporal Fourier modes  less than $N$, we can assume that
\begin{equation} \label{deriveestemps}
 |\pa_t^s \pa^\alpha_x \pa^\beta_y u(t,x,y)| \:  \le \: C_s \,     | \pa^\alpha_x \pa^\beta_y u(t,x,y)| , \quad \forall t,x,y, \quad \forall \alpha, \beta.  
 \end{equation}

Our key idea to get uniqueness for equation \eqref{prandtl3} is to write it as   
$$ \pa_x L u = -\pa_t u +  \pa_y^2 u, \quad \mbox{ where } \:  L u \: :=  \: u_s u \: - \:  u'_s \int_0^y u,   $$
and to  use $Lu$ and some variants of it as multipliers. Denoting by $( \, | \, )$ the scalar product in $L^2(\RR_+)$, we have for $k=0,1,2$:  
\begin{equation} \label{eqLu}
 \frac{1}{2}\pa_x \| \pa_y^k Lu \|^2_{L^2(\RR_+)} = (-\pa_t \pa_y^k u | \pa_y^k Lu ) \: + \:    (\pa^{k+2}_y u | \pa_y^k Lu) \: := \: I_{k,1} \: + \: I_{k,2}. 
 \end{equation}
We estimate those two terms separately, starting with $I_{k,1} = I_{k,1}(t,x)$. 
Let  $M$ be a constant such that
 $u_s = U$ for all $y \ge M$. We claim that : $ \forall \eps > 0,  \: \forall k =0,1,2, $
\begin{equation} \label{I1k}
 \int_{\T} I_{k,1}(t) \, dt  \: \le \: C_\eps \int_\T \left( \| u(t) \|^2_{L^2(0,M)} \: + \:   \| \pa_y u(t) \|^2_{L^2(0,M)} \right) dt \: + \: \eps \int_\T \| \pa_y^2 u(t) \|^2_{L^2(0,M)} dt.   
\end{equation}
Note that, for clarity,  we omit to indicate the dependance on  $x$ of the various quantities. For brevity, we  just  prove our claim  in  
 the case $k=2$, which is the most involved: one has 
 \begin{align*}
 I_{2,1} \:  = \:  -(\pa_t \pa_y^2 u | \pa_y^2 Lu ) \: & = \:  -\int_{\RR_+}  \pa_t \pa_y^2 u  \left( u_s \pa^2_y u + u'_s \pa_y u - u^{(3)}_s \int_0^y u - u''_s u \right) dy   \\
& = \:  -\frac12\frac{d}{dt}  \int_{\RR_+}  u_s | \pa^2_y u_s |^2  dy -   \int_{\RR_+}   \pa_t \pa_y^2 u  \left(  u'_s \pa_y u - u^{(3)}_s \int_0^y u - u''_s u \right) dy.   
\end{align*}
Integration in time makes the first term at the r.h.s. vanishes. The other  terms involve derivatives of $u_s$, so that the integrals over $\RR_+$ can be replaced by integrals over $(0,M)$. Using Cauchy-Schwarz inequality, we end up with 
\begin{align*}
 \int_\T I_{2,1} \: & \le \:  C \int_\T    \|  \pa_t \pa_y^2 u  \|_{L^2(0,M)} \left( \| \pa_y u \|_{L^2(0,M)} \: + \: \| \int_0^y u \|_{L^2(0,M)} \: + \: \|  u \|_{L^2(0,M)} \right) \\
 & \le \:  C \int_\T   \|  \pa_t \pa_y^2 u  \|_{L^2(0,M)} \left( \| \pa_y u \|_{L^2(0,M)} \: + \: (1+M^2/2) \|  u \|_{L^2(0,M)} \right) \\
& \le \:     \eps  \int_\T    \|   \pa_y^2 u  \|_{L^2(0,M)}^2   \:   + \: 
  C_\eps  \int_\T  \left( \| \pa_y u \|_{L^2(0,M)} \: + \: \|  u \|_{L^2(0,M)} \right)^2   
  \end{align*}
where  Young's inequality and \eqref{deriveestemps} were used for the last inequality. This proves the claim.

\medskip
We now turn to the estimate of 
$I_{k,2}$.
Let $m > 0$ such that $u_s(y) = y$ for $y \le m$.  We claim that: for all $k=0,1,2$, 
\begin{equation} \label{I2k}
 I_{k,2} \: \le  \: C   \left( \| u \|^2_{L^2(0,M)} \: + \:   \| \pa_y u \|^2_{L^2(0,M)} \: + \:  \| \pa_y^2 u \|^2_{L^2(m,M)} \right).    
 \end{equation}
Again, we omitted the $t,x$ dependence in the notations. These inequalities follow from simple  integration by parts.  As before, we consider the case $k=2$, the other ones being easiest. We have
\begin{align*}
 I_{2,2} \:  = \:   \left( \pa_y^4 u | \pa_y^2 Lu \right) \: & = \:  \int_{\RR_+}  \pa_y^4 u  \left( u_s \pa^2_y u + u'_s \pa_y u - u^{(3)}_s \int_0^y u - u''_s u \right) dy   \\ 
& = \:  -\int_{\RR_+}  \pa_y^3 u  \, \pa_y \left( u_s \pa^2_y u + u'_s \pa_y u \right) dy \: -  \:   \int_{\RR_+} \pa_y^2 u   \, \pa^2_y \left( u^{(3)}_s \int_0^y u +u''_s u  \right) dy  
\end{align*}
through  integrations by parts. The last  term at the r.h.s involves derivatives of $u'_s$, so is  for $m \le y \le M$. It also involves up to second order derivatives of $u$ only.  Thus
$$  I_{2,2}  \: \le \:  -\int_{\RR_+}  \pa_y^3 u  \pa_y \left( u_s \pa^2_y u + u'_s \pa_y u \right) dy \: + \:  C \sum_{k=0}^2 \| \pa^k_y u \|^2_{L^2(m,M)}  . $$
As regards the first term, we have
\begin{align*} 
- & \int_{\RR_+}  \pa_y^3 u  \, \pa_y \left( u_s \pa^2_y u + u'_s \pa_y u \right) dy  \:  = \: - \int_{\RR_+}  \pa_y^3 u \left( u_s \pa^3_y u + 2 u'_s \pa^2_y u + u''_s  \pa_y u \right) \\
& \le \:   -\int_{\RR_+}  \pa_y^3 u  \left(2 u'_s \pa^2_y u + u''_s  \pa_y u \right)  \: = \: -\int_{\RR_+} u'_s \pa_y (\pa^2_y u)^2 + \int_{\RR_+} \pa^2_y u  \pa_y (u''_s  \pa_y u ) \\
&  = \: \int_{\RR_+} u''_s  (\pa^2_y u)^2 +  \int_{\RR_+}  \pa^2_y u  \pa_y (u''_s  \pa_y u ) \: \le \: C  \sum_{k=1}^2 \| \pa^k_y u \|^2_{L^2(m,M)}. 
\end{align*}
This proves our claim. 

\medskip
We must now link the Sobolev norms of $u$ and $Lu$. This is the purpose of the 
\begin{lemma} \label{lemmaLu}
Let $u = u(y) \in H^2(0,M)$, $u\vert_{y=0} = 0$. Then,  
$$ \| u \|_{H^1(0,M)} \: \le \: C \, \| L u \|_{H^2(0,M)} $$
and 
 $$ \| u \|_{H^2(m,M)} \: \le \: C_m \, \| L u \|_{H^2(0,M)} $$
\end{lemma}

\noindent
{\em Proof of the lemma.} 
Note that by our assumptions on $u_s$, $Lu$ belongs to $H^2(0,M)$, and satisfies $Lu\vert_{y=0}= \pa_y  Lu\vert_{y=0} = 0$. Denoting $f:= Lu$, we have an explicit representation of $u$ in terms of $f$, by solving the first order ODE $Lu = f$:
\begin{align*}
 u(y) \: & = \: u'_s(y)  \int_0^y \frac{f(t)}{u_s(t)^2} dt \: + \:   \frac{f(y)}{u_s(y)} \\
 & = \:   u'_s(y) \int_0^m \frac{f(t)}{t^2} dt  \: + \:    u'_s(y) \int_m^y \frac{f(t)}{u_s(t)^2} dt +  \frac{f(y)}{u_s(y)} \: := \: u_1(y) \: + \: u_2(y) \: + \: u_3(y).  
\end{align*}
Clearly,  $\| u_2 +  u_3 \|_{H^2(m,M)} \: \le \: C_m \| f \|_{H^2(0,M)} $. 
As regards $u_1$, we compute 
$$ \int_0^m \frac{f(t)}{t^2} dt \: = \:  \int_0^m   \frac{f'(t)}{t} dt  - \frac{f(m)}{m} $$  
so that 
$$ \| u_1 \|_{H^2(0,M)} \: \le \: C  \left| \int_0^m \frac{f(t)}{t^2} dt \right| \: \le \: C' \left( \left( \int_0^m \left| \frac{f'(t)}{t} \right| ^2 \right)^{1/2} \: + \:  \| f \|_{L^\infty(0,M)} \right)   \: \le \: C'' \| f \|_{H^2(0,M)}, $$ 
using Hardy and Sobolev inequalities. We thus obtain the second inequality. 
For the proof of the first inequality, 
it remains to control $u$ on the interval $(0,m)$. 
It satisfies there the equation  
$ y u - \int_0^y u  = f$,
 which gives $y u'  = f'$, and from Hardy's inequality, we deduce that $ \| u' \|_{L^2(0,m)} \: \le \: \| f \|_{H^2(0,m)}$. This ends the proof of the lemma.

\bigskip
Combining this lemma with \eqref{eqLu}, \eqref{I1k} and \eqref{I2k}, we obtain that the solution $u$ of \eqref{prandtl3} with boundary data $u\vert_{x=0} = 0$ satisfies 
\begin{equation} \label{estimateLu}
 \| Lu(x)\|^2_{L^2_t(\T; H^2_y(\RR_+))}  \: \le \:  C_\eps \,  \int_0^x \| Lu(x') \|^2_{L^2_t(\T; H^2_y(\RR_+))}  dx' \: + \: \eps \| \pa^2_y u \|^2_{L^2((0,x)\times \T \times \RR_+)}  
 \end{equation}
for all $\eps > 0$ ($\eps=1$ will be enough here). We still need to control the $L^2$ norm of $\pa^2_y u$. Therefore, we differentiate  \eqref{prandtl3} with respect to $y$, and write the resulting equation as
$$ u_s \pa_x L' u =  \pa^3_y u -\pa_t \pa_y u, \quad L'u \: := \:  \pa_y u - \frac{u''_s}{u_s} \int_0^y u. $$
Note that the second term defining $L'u$ has no singularity at zero, as $u''_s$ vanishes identically for $y \le m$. Multiplying this equation by $L'u$, and integrating in $t,x,y$ we obtain 
$$ \frac12 \int_{\T \times \RR_+} u_s |L' u(x) |^2  = \int_{(0,x) \times \T \times \RR_+} (\pa^3_y u - \pa_t \pa_y u ) L' u.  $$ 
One has easily, still with \eqref{deriveestemps}
$$ \int_{(0,x) \times \T \times \RR_+}  (- \pa_t \pa_y u ) L' u \: \le \: C \sum_{k=0}^1 \| \pa^k_y u \|^2_{L^2((0,x) \times \T \times (0,M))}, $$
whereas a simple integration by parts yield
$$ \int_{(0,x) \times \T \times \RR_+} \pa^3_y u \, L' u \: = \: - \int_{(0,x) \times \T \times \RR_+} \!\!\!\!\!\!\! |\pa^2_y u|^2 \: + \: \sum_{k=0}^2 \| \pa^k_y u\|^2_{L^2((0,x) \times \T \times (m,M))} $$
Hence, 
\begin{multline}
 \frac12 \int_{\T \times \RR_+} u_s |L' u(x) |^2 dydt\: 
 + \: \| \pa^2_y u \|_{L^2((0,x)\times )\T \times \RR_+)}  \\ \le \: C \left(  \sum_{k=0}^1 \| \pa^k_y u \|^2_{L^2((0,x) \times \T \times (0,M))} \:  +  \: \| \pa^2_y u\|^2_{L^2((0,x) \times \T \times (m,M))} \right).  
 \end{multline}
 Combining this inequality with Lemma \ref{lemmaLu} implies that 
 $$ \| \pa^2_y u \|^2_{L^2((0,x)\times \T \times \RR_+)}   \: \le \: C \| Lu \|^2_{L^2((0,x) \times \T; H^2_y(\RR_+))} $$
 Together with \eqref{estimateLu}, this leads to
 $$ \| Lu(x)\|^2_{H^2_{t,y}(\T \times \RR_+)} \: \le \: C_\eps \,  \int_0^x \| Lu(x') \|^2_{L^2_t(\T; H^2_y(\RR_+))}   dx'.  $$
By Gronwall lemma, we deduce that $Lu = 0$, and from there that $u =0$. This concludes the proof of the proposition, and in turn  the proof of Theorem \ref{theo3}. 
\end{proof}

\subsection{Final comments}
The  ill-posedness results discussed in this paper  may  give some insight into the classical  results  of  Oleinik on boundary layer theory.  These results, collected in the book \cite{Ole},  establish various well-posedness theorems for the Prandtl equation, under some monotonicity properties for the data with respect to $y$.  Interestingly, the monotonicity assumptions used by Oleinik are different whether the steady case or the unsteady one is considered. Roughly: 
\begin{itemize} 
\item In the steady case, the main hypothesis is that  $u\vert_{x=0}$ be monotonic near $0$, that is $\pa_y u\vert_{x=0}(0) > 0$. 
\item But in the unsteady case, the monotonicity hypothesis is stronger. The data $u\vert_{x=0}$ and $u\vert_{t=0}$ are assumed to be monotonic everywhere, that is $\pa_y u\vert_{t=0}(x,y) > 0$ and $\pa_y u\vert_{x=0}(t,y) > 0$ for all $y > 0$. 
\end{itemize}
In Oleinik's work, these assumptions appear as technical, connected to some special changes of variables  that are different in the steady and unsteady case (these are  the so-called  Von Mises  and Crocco transforms). 

\medskip
From our results, one may believe  that this difference in the steady and unsteady assumptions is more than technical. Indeed, we have proved that as soon as time is involved in the Prandtl system, either through an IVP or a BVP with time periodicity,  a lack of monotonicity somewhere in the flow ($u_s'(a) > 0$ for some  $a > 0$) is enough to trigger a strong instability (at least at the linearized level). On the contrary, as soon as a steady situation is considered, this instability mechanism dissapears. For instance, the space instability developped in this section relies on high frequency time oscillations, and is not there when the steady BVP is considered. Actually,  the  BVP for the steady version of  \eqref{prandtl3}: 
\begin{equation*} 
\left\{
\begin{array}{rlll} 
   u_s \D_x u +  u_s' \, v   - \D^2_y  u   &=& 0,  \quad y > 0,\\
\D_x  u + \D_y v & = & 0,  \quad  y > 0, \\
u = v  &=& 0, \quad  y = 0.   
\end{array}
\right.
\end{equation*}
can be shown to be well-posed for the shear flow $u_s$ considered in Proposition \ref{uniqueBVP}. Indeed,  the estimates established there to prove uniqueness can be used in the same way to establish some {\it a priori estimates}, and from there a well-posedness result (in appropriate functional spaces).  Again, this is coherent with the works of Oleinik, as we assumed in Proposition \ref{uniqueBVP} that $u_s(y) = y$ for small $y$.  
Still in the spirit of this proposition, it would be very interesting to recover the well-posedness results of Oleinik without special changes of variables, that is simply through linearization and/or energy estimates.

\end{document}